\documentclass{amsart}
\usepackage{algorithmic}
\usepackage{algorithm}
\usepackage{graphicx}
\usepackage{subfigure}
\numberwithin{equation}{section}
\theoremstyle{plain}
\newtheorem{definition}{Definition}[section]
\newtheorem{proposition}{Proposition}[section]
\newtheorem{theorem}{Theorem}[section]
\newtheorem{lemma}{Lemma}[section]
\newtheorem{remark}{Remark}[section]

\begin{document}

\title{Perfect Simulation of Determinantal Point Processes}
\author{L. Decreusefond \and I. Flint \and K.C. Low}
\address{Telecom Paristech - LTCI UMR 5141, 46 Rue Barrault, 75634 Paris Cedex 13, France}

\begin{abstract}
Determinantal point processes (DPP) serve as a practicable modeling for many applications of repulsive point processes. A known approach for simulation was proposed in \cite{Hough(2006)}, which generate the desired distribution point wise through rejection sampling. Unfortunately, the size of rejection could be very large. In this paper, we investigate the application of perfect simulation via coupling from the past (CFTP) on DPP. We give a general framework for perfect simulation on DPP model. It is shown that the limiting sequence of the time-to-coalescence of the coupling is bounded by
$$O\left(\int_{E}{J(x,x) d\nu(x)}\log{\int_{E}{J(x,x) d\nu(x)}}\right).$$ An application is given to the stationary models in DPP.
\end{abstract}

% \begin{keyword}
% {Determinantal Point Processes}, {Coupling from the Past}, {Exact Simulation},{Spatial Birth and Death Process},{Papangelou Conditional Intensity},{M/M/Infinity Queue}}
% \end{keyword}

\maketitle{}

\section{Introduction}
Determinantal point process stems back from 1975 when O. Maachi introduce it as the `fermion' process with repulsive feature on its points. It is only in the last two decades
that Soshnikov(2000) \cite{Soshnikov(2000)} and Shirai and Takahashi(2003)\cite{ShiraiTaka(2003)} studied further on the mathematical properties of the process structure.
Classical modeling of repulsive point processes are in essence through Gibbs point processes. Since their introduction, DPP have found many applications in random matrix theory and quantum physics.

Simulations of determinantal point processes are mostly based on the idea of Hough et al. in \cite{Hough(2006)} and further discussed in details in \cite{LMR12}. The main drawback of the idea is
due to its large rejections in the sampling. On the other hand, in a Markov chain Monte Carlo method, many of the simulation algorithms are based on the long-running constructions of Markov chain that converges to an equilibrium distribution. The difficulties was in determining the number of steps needed to have such a convergence. Thanks to Propp and Wilson \cite{ProppWilson(1996)}, which suggest with the application of coupling theory, we are able to `exactly' simulate a finite state Markov chain with the desired equilibrium distribution. Although perfect simulation is obviously appealing but we share a common drawback with \cite{Hough(2006)} idea: given a DPP kernel with no explicitly known spectral representation, the statistics (Papangelou conditional intensity) involve in the configurations, it is not simple to extract them from. In general numerical techniques such as Fourier expansion are required because of the limitation of analytical results.

This paper serves as the continuation of work on \cite{Hough(2006)} and \cite{LMR12} in simulating DPP; in \cite{Kendall(1998)} uses the CFTP perfect simulation approach on spatial point processes,
while here we analyze the application of CFTP simulation on determinantal point processes and we provide a lower and a upper bound for the coalescence time in general case.

This paper is organized as follows. We start in Section \ref{sec:Pre} by summarizing the basic notation and recall the definition of a point process, including determinantal point process and its Papangelou conditional intensity. Section \ref{sec:simulation} is devoted to details of perfect simulation of DPP via dominated coupling of Markov chains, i.e. CFTP. In Section \ref{sec:RunTime}, we discuss the limiting behavior of the running time of CFTP simulation and provide a bound for the limiting sequence. Lastly in Section \ref{sec:App}, we apply CFTP algorithm to three stationary DPP models defined by the commonly used covariance functions in multivariate statistical analysis, i.e. Gaussian model, Mat{\'e}rn model and the Cauchy model. 

The statistical analysis conducted in this paper are with R (\texttt{spatstat} library by Braddeley and Turner).

\section{Preliminaries}\label{sec:Pre}
For $(E,\mathcal{B},\nu)$ a measure space endowed with a Polish space $E$ and a Radon measure $\nu$ on $E$.
We denote by $\mathcal{X}$ the set of locally finite point configurations in $E$:
\begin{equation*}
\mathcal{X} :=\{ \xi \subset E: |\xi \cap \Lambda|<\infty, \quad \forall \mbox{ compact 	} \Lambda \subseteq E\}
\end{equation*}
\noindent equipped with the $\sigma$-field
\begin{equation*}
\mathcal{F} := \sigma (\{ \xi \in \mathcal{X}: |\xi \cap \Lambda|=i\}, \quad i\ge 0, \forall \mbox{ compact 	} \Lambda\subseteq E)
\end{equation*}
\noindent where $|X|$ denotes the cardinality of a set $X$.

A point process $X$ on $E$  is then a random point configuration, i.e. a random integer-valued Radon measure on $E$ and $X(\Lambda)$ represents the number of points that fall in $\Lambda$. If $X(\{x\})\in \{0,1\}$ a.s. for all $x\in E$, then $X$ is called simple. 

The joint intensities $\rho_k$ of a simple point process $X$ is the intensity measure of the set of ordered k-tuples of distinct points of $X$, $X^{\wedge k}$.
More precisely, for any family of mutually disjoint subsets $\Lambda_i \subseteq E$, a simple point process $X$ w.r.t. the measure $\nu$, denote $\rho_k:E^k \rightarrow [0,\infty)$ for $k\ge 1$:
\begin{equation*}
\mathrm{E}\left[\prod^k_{i=1}X(\Lambda_i)\right]=\int_{\prod_i{\Lambda_i}}{\rho_k(x_1,\dots,x_k) d\nu(x_1)\dots d\nu(x_k)}.
\end{equation*} 
We assume that $\rho_k(x_1,\dots,x_k)$ vanish if $x_i=x_j$ for some $i\neq j$, see \cite{Peres}.
Let $\mathbb{K}(x,y): E^2 \rightarrow \mathbb{C}$ be a measurable function, locally square integrable on $E^2$.
\begin{definition}
Determinantal (fermion) point process with kernel $\mathbb{K}$ is defined to be a simple point process $X$ on $E$ which satisfies:
\begin{equation*}
\rho_k(x_1,\dots, x_k) = \mathrm{det}(\mathbb{K}(x_i,x_j))_{1\le i,j\le k}, \quad 
\end{equation*}
$\forall k\ge 1 \mbox{ and } x_1,\dots, x_k \in E$, $\rho_k$ w.r.t the measure $\mu$
\end{definition}

\begin{definition}
For a function $f\in L^2$ defined on a compact support $\Lambda$, an integral operator $\mathcal{K}:L^2(E,\nu)\rightarrow L^2(E,\nu)$
corresponding to $\mathbb{K}$ is defined such that: 
\begin{equation*}
\mathcal{K}f(x)=\int_E \mathbb{K}(x,y) f(y) d\nu(y), \mbox{ for a.e. } x\in E
\end{equation*}
and the associated bounded linear operator $\mathcal{K}_\Lambda$ on $L^2(\Lambda,\nu)$ as:
\begin{equation*}
\mathcal{K}_\Lambda f(x)=\int_\Lambda \mathbb{K}(x,y) f(y) d\nu(y), \mbox{ for a.e. } x\in \Lambda.
\end{equation*}
\end{definition}

By spectral theorem, for a compact and self-adjoint operator $\mathcal{K}_\Lambda$, there is an orthonormal basis $\{\varphi_i^\Lambda\}_{i\ge 1}$ of eigenfunctions
for $\mathcal{K}_\Lambda$ on $L^2(\Lambda,\nu)$. Consequently, the kernel $\mathbb{K}$ has a spectral representation:
\begin{equation}\label{eq:SpectralRepresentation}
\mathbb{K}_\Lambda(x,y)=\sum_{i\ge 1}{\lambda_i^\Lambda\varphi_i^\Lambda(x)\overline{\varphi_i^\Lambda(y)}} \quad \mbox{for } x,y \in \Lambda
\end{equation}
\begin{definition} A bounded linear operator $\mathcal{K}_\Lambda$ is said to be of trace class whenever for a complete orthonormal basis $\{\varphi_i^\Lambda\}_{i\ge 1}$ of $L^2(\Lambda,\nu)$,
\begin{eqnarray*}
	\mathbb{K}_\Lambda(x,y)&=&\sum_{i\ge 1}{\lambda_i^\Lambda\varphi_i^\Lambda(x)\overline{\varphi_i^\Lambda(y)}} \quad \mbox{for } x,y \in \Lambda\\
	\mbox{and}\qquad \quad \sum_{i\ge 1}{|\lambda_i^\Lambda|} &<&\infty.
\end{eqnarray*}
Further we define:
\begin{equation*}
	\mathrm{Tr}~\mathcal{K}_\Lambda := \sum_{i\ge 1} \lambda_i^\Lambda
\end{equation*}
\end{definition}
Throughout this paper, we assume that the integral operator of a determinantal point process satisfies the following hypothesis:

\textbf{Hypothesis (H1):} $\mathcal{K}$ is self-adjoint and locally of trace class, and
its spectrum is contained in $[0,1[$ , i.e. $0\le \mathcal{K} \le Id$ in the operator ordering,
and $\|\mathcal{K}\|<1$, where $Id$ denotes the identity operator on $L^2(E,\nu)$.	

Suppose that a DPP $X$ defined on $E$ with its kernel $\mathbb{K}(x,y)=\sum_{i\ge 1}{\lambda_i \varphi_i(x)\overline{\varphi_i(y)}}$, then it has a pleasant property that the size of its configuration is an infinite sum of independent Bernoulli random variables with parameters equal to the eigenvalues, see \cite{Hough(2006)}. Consequently we have:
\begin{equation}\label{eq:DPPmeanVar}
	\mathrm{E}[|X|]=\sum_{i=1}^{\infty}{\lambda_i} \qquad \mathrm{Var}[|X|]=\sum_{i=1}^{\infty}{\lambda_i(1-\lambda_i)}
\end{equation}

\begin{definition} For a trace-class operator $\mathcal{K}$, the Fredholm determinant is defined by:
\begin{equation*}
\mathrm{det}(I-\mathcal{K}) = \exp \left(\sum^{\infty}_{n=1}{\frac{(-1)^{n-1}}{n} \mathrm{Tr}(\mathcal{K}^n)}\right).
\end{equation*}
\end{definition}

From \cite{Soshnikov(2000)}, given an arbitrary compact set $\Lambda \subseteq E$, the Janossy density $j_{\Lambda}(\cdot)$ of a DPP configuration $\xi$ in $\Lambda$ is given by:
\begin{equation}\label{eq:DPPJanossy}
j_{\Lambda}(\xi) = \mathrm{det}(I-\mathcal{K}_\Lambda) \mathrm{det}J_\Lambda(\xi),
\end{equation}
where we define $J_\Lambda:$ $\Lambda^2\rightarrow \mathbb{C}$ by
\begin{equation}\label{eq:Jmatrix}
J_\Lambda(x,y)=\sum_{i\ge 1}{\frac{\lambda_i^\Lambda}{1-\lambda_i^\Lambda}\varphi_i^\Lambda(x)\overline{\varphi_i^\Lambda(y)}} \quad \mbox{for } x,y \in \Lambda,
\end{equation}
and given a configuration $\xi=\{x_1,x_2,\dots,x_n\}$, $J_\Lambda(\xi)$ is a $n\times n$ matrix with
\begin{equation*}
J_\Lambda(\xi)_{(i,j)}=J_\Lambda(x_i,x_j), \qquad \forall~0<i,j\le n.	
\end{equation*}
We define $\mathrm{det}J_\Lambda(\emptyset,\emptyset)$ = 1. For details on Janossy density, see \cite{Daley}.

In the following, we summarize some facts concerning DPP Papangelou (conditional) intensity which characterize the local dependence of particles.
See \cite{Daley2}\cite{Papangelou} for details on Papangelou intensity.

\begin{definition}\label{def:PCI}
Given a compact set $\Lambda \subseteq E$ and a simple finite point process $X$ defined on $E$ with its Janossy density $j_\Lambda$, 
the Papangelou intensity of $X$ in $\Lambda$ defined as:
\begin{equation*}
c_\Lambda(\xi,x) = \frac{j_\Lambda(\xi\cup x)}{j_\Lambda(\xi)}, \quad \mbox{where }\xi \in \mathcal{X},~x \in \Lambda \backslash \xi.
\end{equation*}
If $j_\Lambda(\xi) = 0$ then $c(\xi,x) = 0$.
\end{definition}
For a Poisson point process with intensity function $\rho$, $c(\xi,x)$ is independent of $\xi$, i.e. $c(\xi,x) = \rho(x)$.
In general, $X$ is characterized as:
\begin{eqnarray}
\mbox{attractive if}: && c(\xi,x) \le c(\eta,x) \mbox{ whenever } x \notin \xi \subseteq \eta \nonumber \\
\mbox{repulsive if}: && c(\xi,x) \ge c(\eta,x) \mbox{ whenever } x\notin \xi \subseteq \eta \nonumber
\end{eqnarray}

\begin{definition}[\cite{Daley2}]
For any simple point process $X$ defined on $E$, where $\psi \in \mathcal{B}$, $\zeta \in \mathcal{F}$ and $\mu$ is the distribution of $X$ on $\mathcal{X}$, 
its modified Campbell measure $C_{\mu}$ on the product space $(\mathcal{X}\times E,\mathcal{F} \otimes \mathcal{B})$:
\begin{equation*}
C_{\mu}( \zeta \times \psi ) = \int_\psi \sum_{x\in X}{\delta_{\{(X\backslash x, x)\in (\zeta \times \psi)\}}} ~\mu(dX),
\end{equation*}
where $\delta$ is the Dirac measure.
\end{definition}

Now, we make an additional assumption that $C_{\mu} \ll \nu \otimes \mu$. From the Definition~\ref{def:PCI} and (\ref{eq:DPPJanossy}),
the Papangelou intensity of DPP then follow from the following proposition which we borrow from \cite{Georgii}.
\begin{proposition} Given $C_{\mu} \ll \nu \otimes \mu$, $\forall \mbox{ compact } \Lambda \subseteq E$, DPP Papangelou intenstity is given by:
\begin{equation*}
c_{\Lambda}(\xi,x) = \frac{\mathrm{det}J_\Lambda(\xi \cup x)}{\mathrm{det}J_\Lambda(\xi)}, \quad \mbox{where }\xi \in \mathcal{X}\mbox{ and } x \in \Lambda \backslash \xi.
\end{equation*} 
If $\mathrm{det}J_\Lambda(\xi)=0$, define $c_\Lambda(\xi,x) = 0$.
\end{proposition}
Since $J_\Lambda(\cdot,\cdot)$ is a positive semi-definite matrix, which can be written in the form of:
\begin{displaymath}
J_\Lambda(\xi x,\xi x) =
\begin{pmatrix}
A_{\xi,\xi} & A_{\xi,x}\\
A_{x, \xi}  & A_{x,x}
\end{pmatrix},\quad \mbox{ where } A_{x,\xi} = A_{\xi,x}^\dagger.
\end{displaymath}

Suppose that $\mathrm{det}A_{\xi,\xi}\neq 0$, then
\begin{displaymath}
\begin{pmatrix}
A_{\xi,\xi} & A_{\xi,x}\\
A_{x, \xi}  & A_{x,x}
\end{pmatrix}=
\begin{pmatrix}A_{\xi,\xi} & 0\\ A_{x, \xi}& I\end{pmatrix} \begin{pmatrix}I& (A_{\xi,\xi})^{-1} A_{\xi,x}\\ 0& A_{x,x} - A_{x, \xi} (A_{\xi,\xi})^{-1} A_{\xi,x}\end{pmatrix}
\end{displaymath}
and
\begin{equation*}
\mathrm{det}(J_\Lambda(\xi x,\xi x)) = \mathrm{det}(A_{\xi,\xi})\mathrm{det}( A_{x,x}-A_{x, \xi}(A_{\xi,\xi})^{-1}A_{\xi,x}).
\end{equation*}

From Hypothesis \textbf{(H1)}, we know:
\begin{equation*}
\mathrm{det}(A_{\xi,x}^\dagger(A_{\xi,\xi})^{-1}A_{\xi,x})\ge 0
\end{equation*}
and the following Lemma~\ref{lemma:CPIupperbound} follows.

\begin{lemma}\label{lemma:CPIupperbound} Papangelou intensity of a DPP is upper bounded:
\begin{equation}
\max_{\xi \in \mathcal{X},x \in \Lambda} c_\Lambda(\xi,x) = \max_{\xi \in \mathcal{X},x \in \Lambda}\frac{\mathrm{det}(J_\Lambda(\xi \cup x))}{\mathrm{det}(A_{\xi,\xi})} \le \max_{x \in \Lambda}J_\Lambda(x,x),
\end{equation}
and for a stationary model DPP:
\begin{equation*}
\max_{\xi \in \mathcal{X},x \in \Lambda} c(\xi,x) \le H.
\end{equation*}
where a stationary model DPP is a DPP with kernel of the form: $\mathbb{K}(x,y)=\mathbb{K}_0(x-y)$.
Thus, its $J_\Lambda(x,x)$ is equal to a constant $H$, for all $x \in \Lambda$.
\end{lemma}

\begin{remark}\label{remark:repulsive} From the above lemma, we have also shown that a DPP is a repulsive point process where its Papangelou intensity: $c(\xi,x) \le c(\eta,x)$ for $x \notin \eta \subseteq \xi$.
\end{remark}

\section{Simulation}\label{sec:simulation}
\subsection{Perfect Simulation via Dominating CFTP}
In this section, we describe the idea of using birth and death process to couple from the past of a continuous time Markov chain $X=\{ X_t: t\ge 0\}$ with values in $\mathcal{X}$.
From here, we restrict all the following point processes to be define on a Polish space $E$. 

A birth and death process $(X_t)_{t\ge 0}$ with birth rate $b$ and death rate $d$ is a homogenous Markovian process defined on \{$\mathbb{N} \cup {0}$\}.
Birth rate $b$ and death rate $d$ are non-negative functions defined on $\mathcal{X} \times E$.
The process $X_t$ is right-continuous and piecewise constant except at jump times $T_1<T_2<\dots$, where we define:
\begin{equation*}
B(\xi)=\int_E{b(\xi,x)d\nu(x)}, \quad \delta(\xi)=\sum_{x\in \xi}{d(\xi\backslash x, x)} \mbox{ for }\xi \neq\emptyset \mbox{, else }\delta(\emptyset)=0
\end{equation*}
and
\begin{equation*}
T_{m+1}-T_m \sim \mathrm{Exponential}(B(\xi)+\delta(\xi))
\end{equation*}

By conditioning on $T_{m+1}$, a birth occurs with probability:
\begin{equation*}
\frac{B(\xi)}{B(\xi)+\delta(\xi)}	
\end{equation*}
and a death with probability:
\begin{equation*}
\frac{\delta(\xi)}{B(\xi)+\delta(\xi)}.	
\end{equation*}

A $M/M/\infty$ queue is a birth and death process used to describe a multi-server queuing model, where its arrival rate is defined: $\lambda:=\int_E{b(\xi,x)d\nu(x)}$ and service rate is defined: $\mu:=d$.
As the total service rate of a $M/M/\infty$ queue is proportional to its size, the process is always stable.

A Markov Chain Monte Carlo method is to let the birth rate $b$ equals to the Papangelou intensity of the desired point process $X$
and the death rate $d=1$, then the Markov chain $\hat{X}=\{\hat{X}_t: t\in \mathbb{R}\}$ constructed converges to the distribution of $X$. 
To obtain a perfect simulation, Kendall and M{\o}ller \cite{KendallMoller(2000)} introduce the dominating process $\{D_t : t\in \mathbb{R}\}$ as a spatial birth and death process with death rate $d=1$ and birth rate $b=H$, where $H$ is the upper bound of the Papangelou intensity of the desired point process $X$. The intuition is to introduce a coupling (pairs of Markov chains, $L_{T_i}$ and $U_{T_i}$) from the past,
deriving from $D$ that respect a partial ordering on the state space $\mathrm{X}$ corresponding to $\hat{X}_t$ under time-evolution, i.e.:
\begin{eqnarray}
L_t \subseteq \hat{X}_t &\subseteq& U_t \subseteq D_t \qquad \forall s \le t \le 0 \label{eq:partialorder},\\
L_t = U_t  &\mbox{if}& L_s = U_s \qquad \forall s \le t \le 0 \label{eq:coalescentorder},
\end{eqnarray}

Next, we define a marking process $\{M_t : t\in \mathbb{R} \}$ which is independent from $D$ and such that we can adaptively contruct $\hat{X}$, $L$ and $U$ as functionals of $(D,M)$ .
We denote the process $\hat{X}^n$ as as the constructed Markov chain that begun at time $-n<0$. Suppose that we have coalescence between $L^n$ and $U^n$ during the progression from time $-n$ to $0$, 
then the exact equilibrium is attained by $\hat{X}^n$ at time $0$, where we have $L^n_0=\hat{X}^n_0=U^n_0$.
\begin{remark} Technically, we have to progressively increase $n$ if coalescence failed at time $0$, however, in \cite{ProppWilson(1996)} suggested that it is efficient to iteratively doubling $n$, and we let $n \in \{\frac{1}{2}, 1,2,4,\cdots\}$. In later section Section \ref{sec:RunTime}, we will show that the running time is bounded in limiting sense. 
Note that, $D$ is extended for each doubling of $n$, i.e. $D_{[-2n,0]}$ is computed from extension of $D_{[-n,0]}$.
\end{remark}

Following are the configuration of a perfect simulation. For each jump times $\{T_1<T_2<\dots\}$ in the dominating process $D$:\\
if there is a birth of point $x \in \Lambda \backslash D_{T_{i-1}}$ at time $T_{i}$, we set:
\begin{equation}\label{eq:configuration}
  \begin{array}{l l}
U_{T_{i}} := U_{T_{i-1}} \cup \{x\}&\mbox{if}\qquad M_{T_{i}} \le \frac{c(L_{T_{i-1}},x)}{H},\\
U_{T_{i}} := U_{T_{i-1}} &\mbox{otherwise.}\\
L_{T_{i}} := L_{T_{i-1}} \cup \{x\} &\mbox{if}\qquad M_{T_{i}} \le \frac{c(U_{T_{i-1}},x)}{H},\\
L_{T_{i}} := L_{T_{i-1}} &\mbox{otherwise.}
  \end{array}
\end{equation}
From Remark~\ref{remark:repulsive}, we know that $c(L_{T_{i-1}},x)\ge c(U_{T_{i-1}},x)$.\\
On the other hand, if there is a death of point $x \in D_{T_{i-1}}$ at time $T_{i}$, we configure:
$U_{T_{i}} := U_{T_{i-1}} \backslash \{x\}$ and $L_{T_{i}} := L_{T_{i-1}} \backslash \{x\}$ respectively.

Suppose that the process $(D, M)$ is stationary in time and $\hat{X}^n$, $L^n$ and $U^n$ for $\{n = 1, 2, 4,\dots\}$ are derived adaptively
from $(D,M)$ satisfying (\ref{eq:partialorder}) and (\ref{eq:coalescentorder}), then the following proposition follows immediately from the dominated convergence theorem, see \cite{KendallMoller(2000)}.

\begin{proposition}\label{prop:CFTP} Let $N = \inf \{n \in \{\frac{1}{2},1, 2,\cdots \} : L^n_0 = U^n_0\}$, and set $L^{n}_{-n} = \emptyset$ and $U^{n}_{-n} = D_{-n}$.
If as $t$ tends to infinity, $\hat{X}_t$ converges weakly to an equilibrium distribution $\pi$ and the probability of $D$ visiting $\emptyset$ in the time interval $[0, t]$ converges to $1$,
then almost surely $N < \infty$ and $L^N_0 = U^N_0$ follows the equilibrium distribution $\pi$.
\end{proposition}
\begin{figure}[htbp] %htbp
        \subfigure[at time $T_0$]{%
            \label{fig:SimulateT_N}
            \includegraphics[width=0.31\textwidth,trim=70 20 65 20,clip]{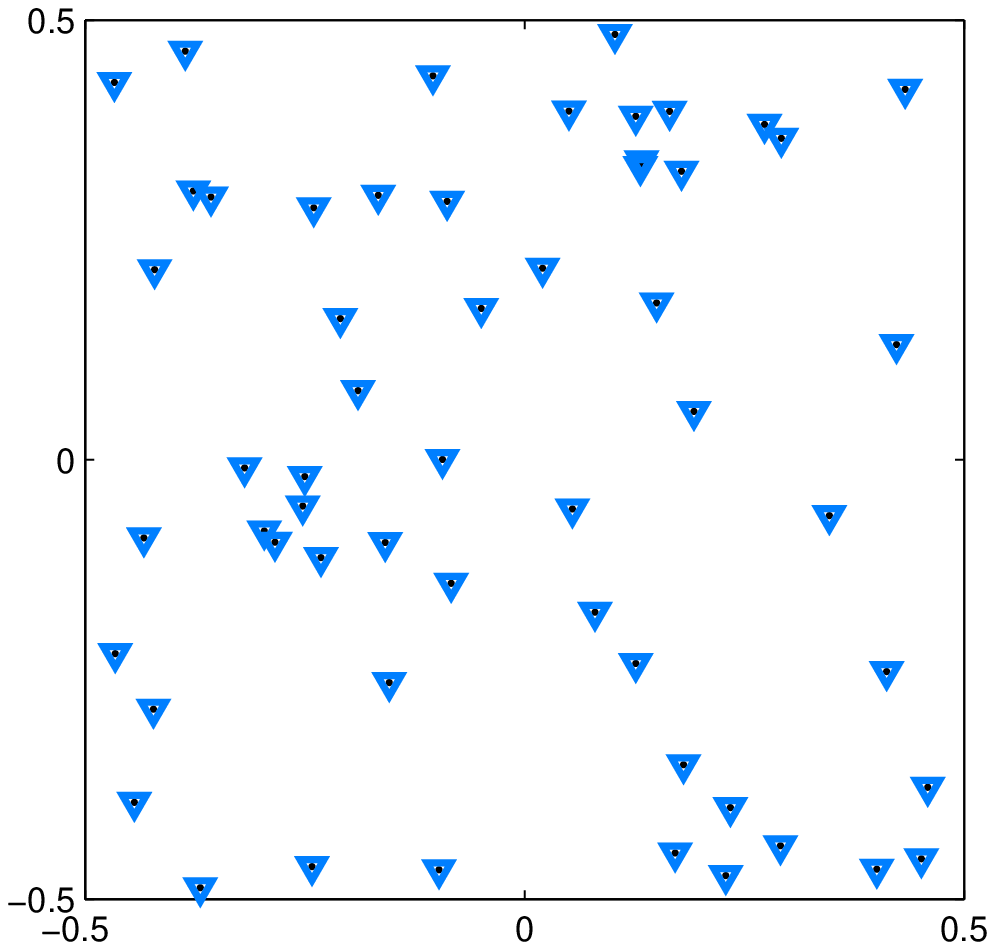}
        }%
        \subfigure[at time $T_{25}$]{%
           \label{fig:SimulateT_N+25}
           \includegraphics[width=0.31\textwidth,trim=70 20 65 20,clip]{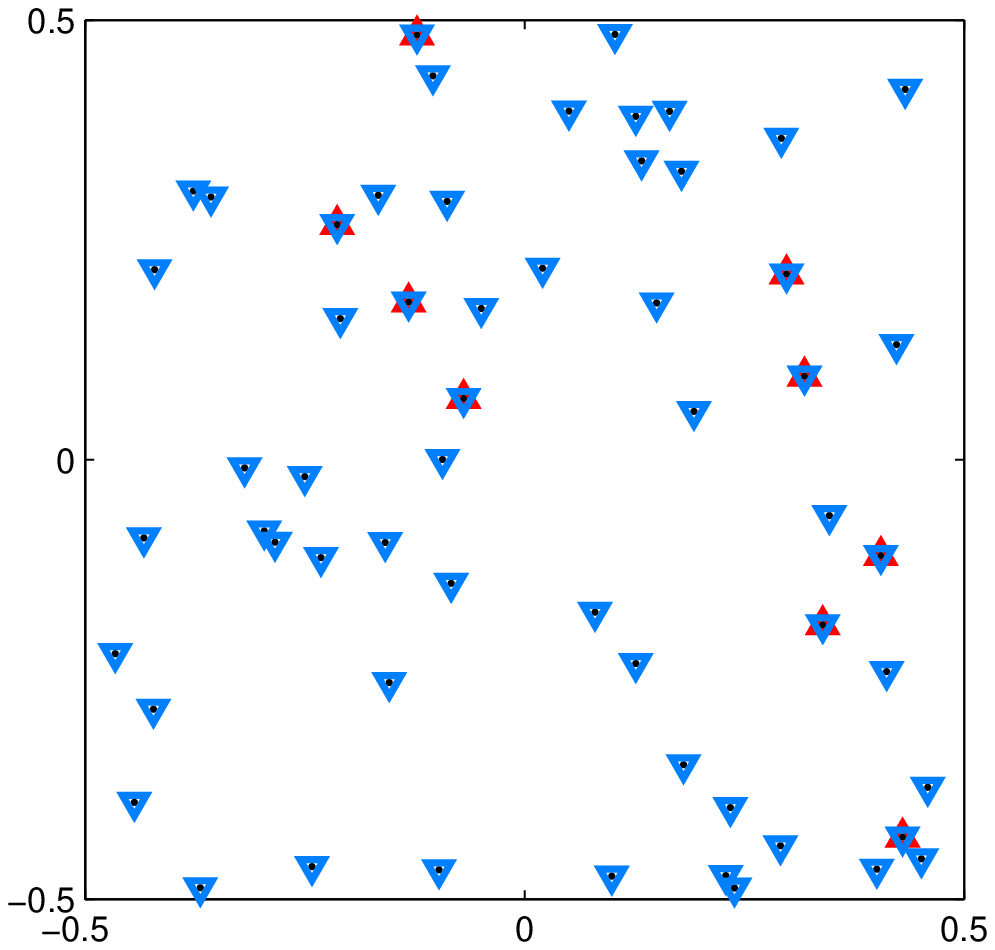}
        }%
        \subfigure[at time $T_{50}$]{%
            \label{fig:SimulateT_N+50}
           \includegraphics[width=0.31\textwidth,trim=70 20 65 20,clip]{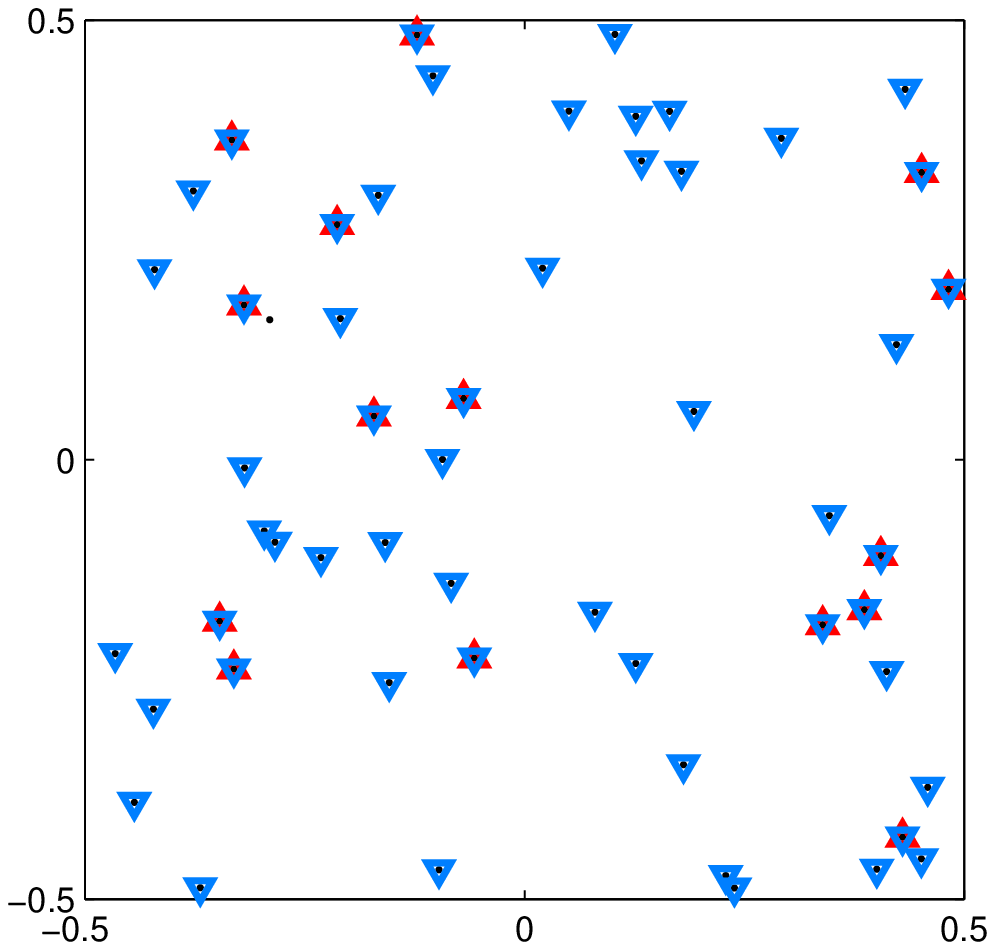}
        }%		
				\caption{CFTP simulations for Gaussian model DPP with $\rho=50$ and $\alpha=0.04$, respectively at time $T_i$, the $i$-th jump time from time $t=-n$.
				Notations: $``\cdot":=D_t$, $``\nabla":=U_t$ and $``\Delta" \mbox{(red)}:=L_t$}
   \label{fig:Simulate}
\end{figure}

\subsection{Simulation of determinantal point processes} In this subsection, we generalize the idea of Kendall and M{\o}ller \cite{KendallMoller(2000)} by relaxing the condition of compactness.
\begin{theorem}\label{theorem:CFTP} Given a DPP $X$ defined on a Polish space $E$ w.r.t. a Radon measure $\nu$. Suppose that the process $D$ started with a Poisson point process (PPP) with intensity measure:
\begin{equation}\label{eq:arrivalrate}
\int_{E}{J(x,x) d\nu(x)}
\end{equation}
and with birth rate $b=J(x,x)d\nu(x)$ and death rate $d=1$. Suppose further, the marking process $M_{T_i} \sim \mbox{Unif}(0,1)$, i.i.d. for each $i$ and independent of $D$.
Then the process $(D, M)$ is stationary in time and almost surely $N<\infty$ and $L^N_0 = U^N_0$ follows the distribution of DPP $X$.
\end{theorem}
\begin{proof}
Stationarity of the process $(D, M)$ follows immediately as the process $D$ started with its equilibrium distribution, i.e. a PPP($J(x,x) d\nu(x)$).

Following, the probability of $D$ visiting $\emptyset$ in the time interval $[0, t]$ converges to $1$ as $t$ tends to infinity. Consequently, from Proposition~\ref{prop:CFTP}
we have $L^N_0 = U^N_0$ follows the distribution of DPP $X$.
\end{proof}

\begin{algorithm}
\caption{Simulation of determinantal point process}
\label{alg:DPP}
\begin{algorithmic}
\STATE \textbf{Sample} $D_0$ from PPP ($\int_{E}{J(x,x) d\nu(x)}$)
\STATE $n \leftarrow 1/2$;
\WHILE{TRUE}
\STATE $D \leftarrow$ BackwardExtend$(D, n)$;
\STATE $[L, U] \leftarrow$ Coupling$(D)$;
\IF{$L_0 == U_0$}
\RETURN $L_0$
\ELSE
\STATE $n \leftarrow n*2$;
\ENDIF
\ENDWHILE
\end{algorithmic}
\end{algorithm}

\begin{algorithm}
\caption{BackwardExtend$(D, n)$}
\label{alg:BackwardExtend}
\begin{algorithmic}
\STATE $j\leftarrow 0;$
\STATE $T(0)\leftarrow n/2;$ \qquad \qquad \COMMENT{$T(0)\leftarrow 0$ if $n=1/2$}
\STATE $\tilde{D}_{T(0)}\leftarrow D_{-n/2};$ \qquad \COMMENT{$\tilde{D}_{T(0)} \leftarrow D_0$ if $n=1/2$}
\WHILE{$T(j) \le n$}
\STATE $T(j+1) \leftarrow T(j) - \log(\mbox{Uniform}(0,1))/(\int_{E}{J(x,x) d\nu(x)}+|\tilde{D}_{T(j)}|)$;
    \IF{$\mbox{Uniform}(0,1) \le (\int_{E}{J(x,x) d\nu(x)})/(\int_{E}{J(x,x) d\nu(x)}+|\tilde{D}_{T(j)}|)$} 
		\STATE	$x \leftarrow $ uniform random point in $\Lambda \backslash \tilde{D}_{T(j)}$;
    \STATE  $\tilde{D}_{T(j+1)} \leftarrow \tilde{D}_{T(j)} \cup x$;
    \ELSE
    \STATE  $x \leftarrow $ uniform random point in $\tilde{D}_{T(j)}$;
    \STATE  $\tilde{D}_{T(j+1)} \leftarrow \tilde{D}_{T(j)}\backslash x$;
    \ENDIF
\STATE    $j\leftarrow j+1$;
\ENDWHILE
\STATE $D_{-t}\leftarrow\tilde{D}_t$ for all $t: n/2 < t\le n$
\RETURN $D$
\end{algorithmic}
\end{algorithm}

\begin{algorithm}
\caption{ \cite{KendallMoller(2000)} Coupling$(D)$}
\label{alg:Coupling}
\begin{algorithmic}
\STATE $L_{-n}\leftarrow \emptyset$;    
\STATE $U_{-n}\leftarrow D_{-n}$;
\FOR{$T_i \leftarrow$ each jump times $T_1<T_2<\cdots$ of $D$ in $]-n:0]$}
    \IF{$D_{T_i} \leftarrow D_{T_{i-1}} \cup x$}
       \STATE $u \leftarrow M_T$;
       \STATE $[L_{T_i},U_{T_i}] \leftarrow$ AddBirth$(L_{T_{i-1}},U_{T_{i-1}},x,u)$;
    \ELSE
       \STATE $x \leftarrow  D_{T_{i-1}} \backslash D_{T_i}$;
			 \STATE $L_{T_i} \leftarrow L_{T_{i-1}} \backslash x$;
			 \STATE $U_{T_i} \leftarrow U_{T_{i-1}} \backslash x$;
    \ENDIF
\ENDFOR
\RETURN $[L,U]$
\end{algorithmic}
\end{algorithm}

In the pseudocode, we begin with setting $n=1/2$ (in later Section~\ref{sec:RunTime}, we will show that we can replace $n=1/2$ with $\log{\int_{E}{J(x,x) d\nu(x)}}$) and construct $D$ backwards in time to $-n$ . Instead of reversing the birth and death rate, the process $D$ is a Glauber process with equilibrium measure $\int_{E}{J(x,x) d\nu(x)}$ and hence we can initiate with state $\tilde{D}_0=D_0=PPP(\int_{E}{J(x,x) d\nu(x)})$ and simulate $\tilde{D}_n$ forwards in time to time $n$ with $b=J(x,x) d\nu(x)$ and $d=1$. Let $D_{-t}=\tilde{D}_t$ for all $t: 0\le t\le n$. We initiate the processes
$L$ and $U$ from time $-n$ with $L_{-n}=\emptyset$ and $U_{-n}=D_{-n}$ respectively, such that (\ref{eq:partialorder}) is satisfied. The implementation of AddBirth$(L_{T_{i-1}},U_{T_{i-1}},x,u)$ in Algorithm~\ref{alg:Coupling} follows directly from (\ref{eq:configuration}) which we replace $H$ as $J(x,x) d\nu(x)$, hence we omit it here. In the algorithm, the coalescence of processes $L_t$ and $U_t$, and their convergence to the target distribution is assured by Theorem~\ref{theorem:CFTP}.

\section{Running time}\label{sec:RunTime}
In the previous sections, we have been concerned with the simulation of a DPP defined on a Polish space $E$. Following, we will provide a bound for the running time of the algorithm through the limiting behavior of a $M/M/\infty$ queue.

As per notations used in Section \ref{sec:simulation}, $|D_t|$ is also called as the $M/M/\infty$ queue with arrival rate: $\lambda=\int_{E}{J(x,x) d\nu(x)}$ and service rate: $\mu=1$.
We defined the process $G_t:= U_t \backslash L_t$ for all $t:\mathbb{R}$.

The coalescence time of $L$ and $U$ (equivalently the running time of the algorithm) is equal to the hitting time of $G_t$ to $\emptyset$.
From (\ref{eq:configuration}), a birth occur in $G_t$ if and only if there is a birth in $U_t$ but not in $L_t$ and a death occur in $G_t$ if and only if a
death occur in $D_t$ and the dead point $x$ is in $U_t$ but not in $L_t$. We obtain the arrival rate, $\lambda^G$ and service rate, $\mu^G$ of $G_t$ as following:
\begin{eqnarray}
\lambda^G \quad=& |E|(c(L_{t},x)-c(U_{t},x)) \label{eq:lambdaG}\\
\mu_t^G \quad=& \frac{|G_t|}{|U_t|} \label{eq:muM}
\end{eqnarray}

Suppose that the process $G_t$ initiated at time $t=t_0$ and we have a sufficiently large $|G_{t_0}|=z$, then after an exponential time, the process goes to size $z+1$ with probability
$\frac{\lambda^G}{\lambda^G + z\mu^G}$ and $z-1$ with probability $\frac{z\mu^G}{\lambda^G + z\mu^G}$. Hence, the next jump is very likely to be a death. We can roughly approximate the
order of hitting time by 
\begin{equation*}
\sum_{i=1}^{z}{\frac{1}{\lambda^G+i\mu^G}} \sim \frac{\log{z}}{\mu^G}
\end{equation*}

A rigorous approach is provided in \cite{Robert(2003)}, as the following Proposition~\ref{prop:HittingTime}.
\begin{proposition} \label{prop:HittingTime} Given a $M/M/\infty$ queue $G_t$ with arrival rate $\lambda$ and service rate $\mu$ (initiated at $t_0$, with $|M_{t_0}|=z$), the hitting time $T_\emptyset$ of $\emptyset$ is of the order $\log z$. Precisely: 
\begin{equation*}
\lim_{z\to\infty}\Pr\left(\left|\frac{T_\emptyset}{\log{z}}-\frac{1}{\mu}\right| \geq \varepsilon\right) = 0.
%\lim_{n\rightarrow \infty} \frac{T_\emptyset}{\log{n}} = \frac{1}{\mu} \quad \mbox{ in probability}
\end{equation*}
\end{proposition}

On the other hand, the hitting time of $G_t$ to $\emptyset$ is at least lower bounded by the time where all the initial points in $G_{t_0}$ perished, accodingly we have the following Proposition.
\begin{proposition}\label{prop:lowerbound} Given a $M/M/\infty$ queue $G_t$ with arrival rate $\lambda$ and service rate $\mu_t$ (initiated at $t_0$, with $|G_{t_0}|=z$ and $\mu_{t_0}=1$), the expected hitting time $\mathrm{E}[T_\emptyset]$ of $\emptyset$ is lower bounded:
\begin{equation}\label{eq:lowerbound}
\ln z \quad \le \quad \sum_{i=1}^{z}{\frac{1}{i}} \quad = \quad \mathrm{E}[T_\emptyset]
\end{equation}
\begin{proof}
Given $\mu_{t_0}=1$, each of the initial point in the configuration $G_{t_0}$ have living period i.i.d. $\mathrm{Exponential} (1)$.
It is not difficult to show that the expected value of the maximum of $z$ exponential random variable, $x_i$ with parameter $1$ is:
\begin{equation*}
\mathrm{E}[\max(x_i)] = \sum_{i=1}^{z}{\frac{1}{i}}  \qquad \mbox{ (harmonic series)}.
\end{equation*}
Hence, (\ref{eq:lowerbound}) is proved.
\end{proof}
\end{proposition}

From Propositions \ref{prop:HittingTime} and \ref{prop:lowerbound}, given $N = \inf \{n \in \{\frac{1}{2},1, 2,\cdots \} : L^n_0 = U^n_0\}$ and the size of the initial configuration of the process $G^N_{-N}$ equal to $z$, we have the following:
\begin{itemize}
	\item the expected hitting time is lower bounded by $\log{z}$.
	\item the hitting time is upper bounded by $z \log{z}$ where $\frac{1}{\mu_t^G} \le z$.
\end{itemize}

Since the size of the configuration for $G^N_{-N}$ is defined to be $|U^N_{-N} \backslash L^N_{-N}| = |D^N_{-N}|$, we now look into the limiting behavior of the re-normalized process of $D_t$.
We introduce a scaling procedure to the underlying space of the process $D_t$ by a factor of $W\in \mathbb{N}$. The scaled process is defined as:
\begin{equation*}
D^W_t := \frac{D_t}{W}	
\end{equation*}
and the underlying space of the scaled process is:
\begin{equation}
E_W: |E_{W}| = W|E|.
\end{equation}

The size of the initial configuration $D_{t_0}^W$ is deterministic, where $|D_{t_0}^W|$ is a non-negative integer such that:
\begin{equation*}
 \lim_{W\rightarrow\infty} \frac{|D_{t_0}^W|}{W} =: x \in \mathbb{R}^+ \cup \{0\}.
\end{equation*}

We have the following theorem from the functional law of large number in \cite{Robert(2003)}.
\begin{theorem}\label{theorem:functionalLLN} Given a $M/M/\infty$ queue $D_t$ with arrival rate $\lambda$ and service rate $\mu$ (initiated at $t=t_0$) and $T\in \mathbb{R}$, define:
\begin{equation}
f(t) := \frac{\lambda}{\mu}+(x-\frac{\lambda}{\mu}) e^{-\mu (t-t_0)},
\end{equation}
then as $W$ tends to infinity,
\begin{equation*}
\sup_{t_0 \le t \le T}  \frac{|D^W_t|}{W} \stackrel{d}{\longrightarrow} \sup_{t_0 \le t \le T} f(t)
\end{equation*}
\end{theorem}

\begin{proposition}\label{prop:SizeBound} Suppose given a $M/M/\infty$ queue $D_t$ defined on a $E$ with arrival rate $\lambda=\int_{E}{J(x,x) d\nu(x)}$ and service rate $\mu=1$.
Suppose further $D_t$ is initiated at $t=t_0$, with $D_{t_0}=\mathrm{PPP}(J(x,x) d\nu(x))$, then we have
\begin{equation}\label{eq:SizeBound}
\sup_{-t_0 \le t \le T}|D_t| = O\left(\max\left(x,\int_{E}{J(x,x) d\nu(x)}\right)\right),
\end{equation}
\begin{proof}
Given $\lambda=\int_{E}{J(x,x) d\nu(x)}$ and $\mu=1$, the equilibrium distribution of the process $D_t$ is $\mathrm{PPP}(J(x,x) d\nu(x))$, i.e.
\begin{equation*}
	|D_t|\sim \mathrm{Poisson}\left(\int_{E}{J(x,x) d\nu(x)}\right)
\end{equation*}
and with the scaling $W$, the scaled process $D_t^W$:
\begin{equation*}
	|D_t^W|\sim \mathrm{Poisson}\left(\int_{E_W}{J(x,x) d\nu(x)}\right)
\end{equation*}
Hence, the limiting sequence of $(\frac{|D^W_t|}{W})$ converges in distribution to $D_t$, i.e.:
\begin{equation*}
\frac{D^W_t}{W} \stackrel{d}{\longrightarrow} D_t.
\end{equation*}
From Theorem~\ref{theorem:functionalLLN}, as $W$ tends to infinity, we have:
\begin{eqnarray*}
\sup_{t_0 \le t \le T}\frac{|D^W_t|}{W} &\stackrel{d}{\longrightarrow}& \sup_{t_0 \le t \le T}{\int_{E}{J(x,x) d\nu(x)}+\left(x-\int_{E}{J(x,x) d\nu(x)}\right)e^{-\mu(t-t_0)}}\\
&=& \sup_{t_0 \le t \le T}{xe^{-\mu(t-t_0)}+\int_{E}{J(x,x) d\nu(x)}(1-e^{-\mu(t-t_0)})}\\
&=& O\left(\max\left(x,\int_{E}{J(x,x) d\nu(x)}\right)\right)
\end{eqnarray*}
Therefore (\ref{eq:SizeBound}) is proved
\end{proof}
\end{proposition}

If we fixed $x=\int_{E}{J(x,x) d\nu(x)}$, applying Proposition \ref{prop:SizeBound} to Proposition~\ref{prop:HittingTime}, we have the stopping time of the algorithm upper bounded by:
\begin{equation}\label{eq:uppebound}
	O\left(\int_{E}{J(x,x) d\nu(x)}\log{\int_{E}{J(x,x) d\nu(x)}}\right),
\end{equation}
and if for all $x\in E$, $J(x,x)=H$ a constant, i.e. a stationary model DPP, then the upper bound is:
\begin{equation*}
	O(H|E|\log{H|E|}).
\end{equation*}
Heuristically, we shall see that the running time of the algorithm is much lower than the bound (\ref{eq:uppebound}) in Section \ref{sec:App}.

\section{Application to examples in stationary models}\label{sec:App}
\begin{figure}[htbp] %htbp
	\centering
        \subfigure[Gaussian Model]{%
            \label{fig:GaussianSample}
            \includegraphics[width=0.31\textwidth, trim=70 20 65 20,clip]{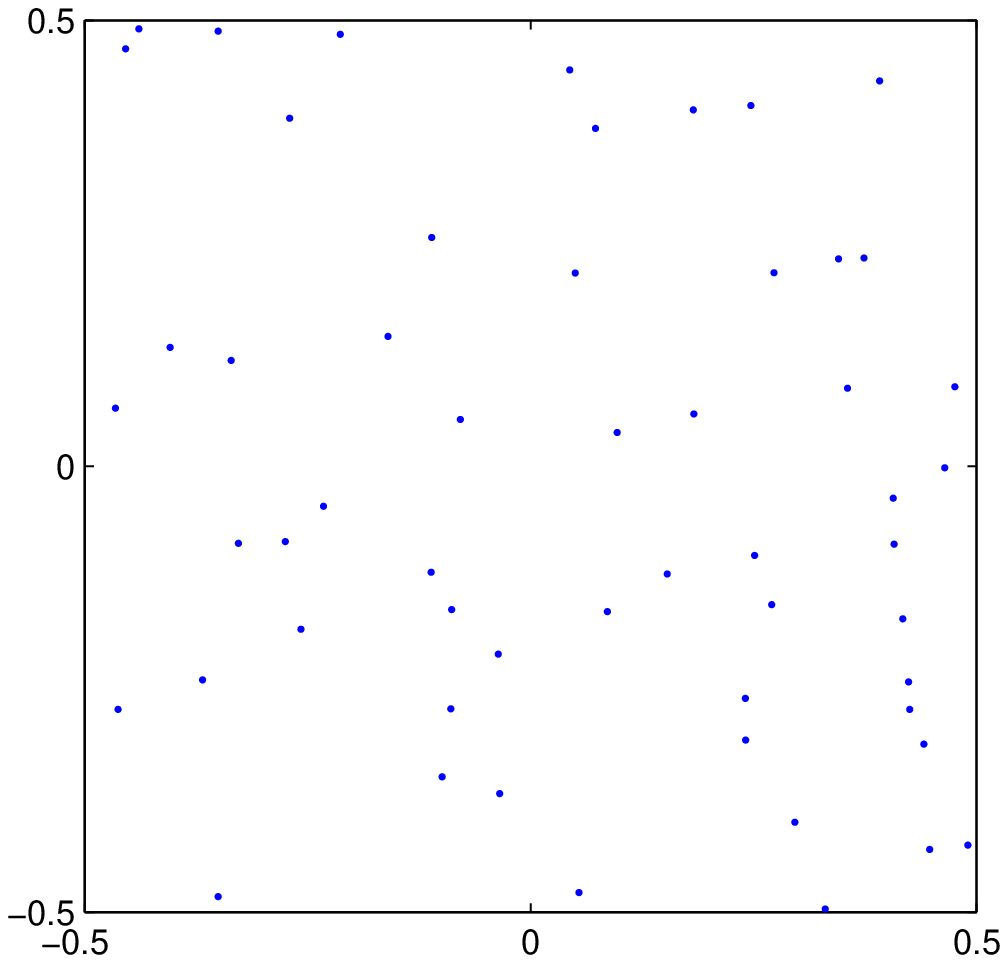}
        }%
        \subfigure[Mat{\'e}rn Model ($\nu=5$)]{%
           \label{fig:MaternSample}
           \includegraphics[width=0.31\textwidth, trim=70 20 65 20,clip]{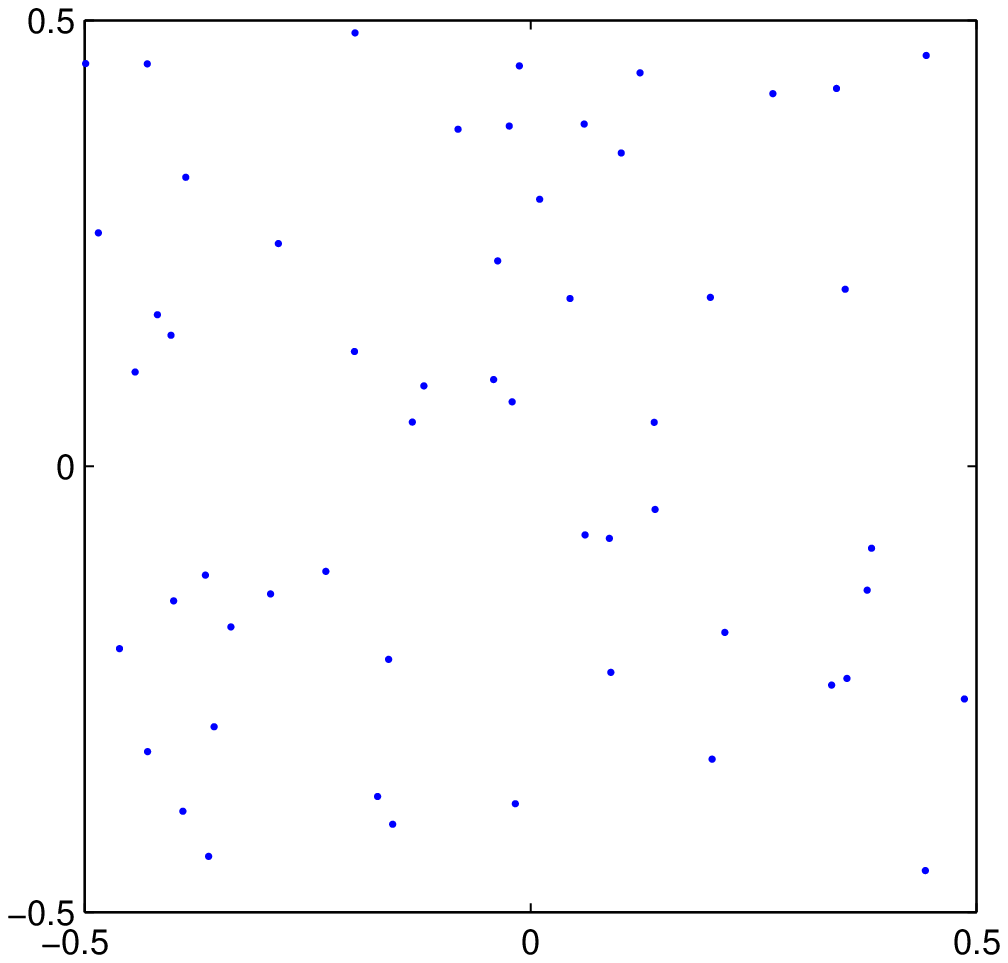}
        }%
        \subfigure[Cauchy model ($\nu=5$)]{%
            \label{fig:CauchySample}
           \includegraphics[width=0.31\textwidth,trim=70 20 65 20,clip]{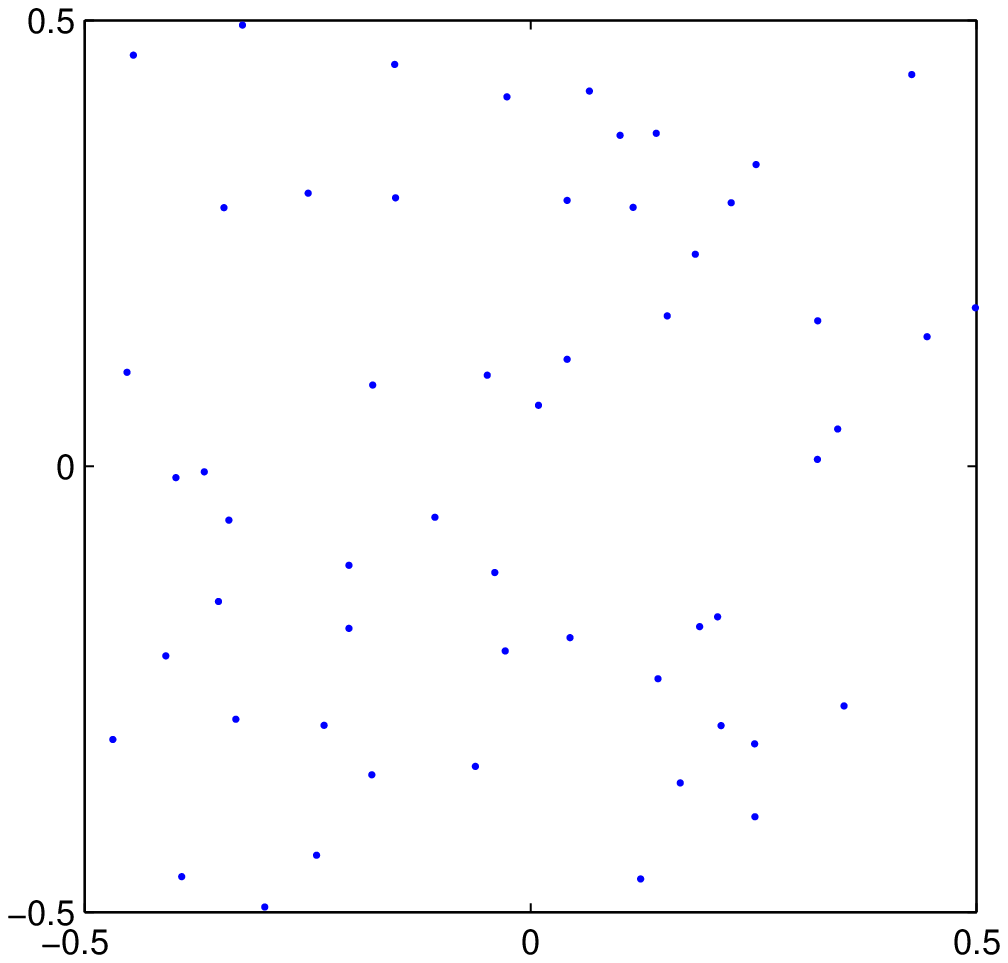}
        }%
				\caption{Configurations of stationary DPPs simulated from CFTP algorithm. $\rho=50$ and $\alpha = \alpha_{\max}/2$ for all the $3$ models.}
   \label{fig:Sample}
\end{figure}
Point processes and random measures that are invariant under shifts in a $d$-dimensional Euclidean space $\mathbb{R}^d$ play a vital role in applications and development of the general theories. Accordingly in this section, we will be focusing on stationary DPP models defined on a compact set $\Lambda=[-1/2,1/2]^2$ in $\mathbb{R}^2$, where we denote $l\in \Delta := \{x-y:x,y\in \Lambda\}$ and 
\begin{equation*}
	\mathbb{K}_0(l) := \mathbb{K}(x,y).
\end{equation*}
Suppose we are given a DPP kernel, it would be ideal if we could explicitly compute the Papangelou intensity (upper bound) in Lemma~\ref{lemma:CPIupperbound}, however (see \cite{Georgii} and \cite{LMR12}) there are only known in a few simple models where the Papangelou intensity can be computed from the spectral representation. In the sequel, we will approximate the kernel
in (\ref{eq:SpectralRepresentation}) by applying Fourier expansion to obtain the Papangelou intensity.

Consider the following orthonormal Fourier basis in $L^2(\Delta)$:
\begin{equation}
	\varphi_k^{\Delta}(l) = e^{2\pi i k \cdot (l)}, \quad k\in \mathbb{Z}^2, l \in \Delta
\end{equation}
where $k \cdot l$ is the dot product of the vectors. We have the Fourier expansion of $\mathbb{K}_0(l)$ as:
\begin{equation*}
	\mathbb{K}_0(l) = \sum_{k\in \mathbb{Z}^2}{\lambda_k^{\Delta} \varphi_k^{\Delta}(l)}.
\end{equation*}
equivalently:
\begin{equation*}
	\mathbb{K}(x,y) = \sum_{k\in \mathbb{Z}^2}{\lambda_k^{\Delta} \varphi_k^{\Delta}(x)\overline{\varphi_k^{\Delta}(y)}},
\end{equation*}
where the Fourier coefficients:
\begin{eqnarray*}
	\lambda_k^{\Delta} &=& \int_{\Delta}{\mathbb{K}_0(l)e^{-2\pi i k \cdot r} dl}\\
	&\approx& \int_{\mathbb{R}^2}{\mathbb{K}_0(r)e^{-2\pi i k \cdot r} dr}\quad =: \quad \lambda_k
\end{eqnarray*}
As $\mathbb{K}_0(l) \approx 0$ when $|l|>1$, we have Fourier transform as an approximation for the Fourier coefficient.

Corresponding to \cite{LMR12}, let $\rho$ defined the intensity of the DPP and we apply the CFTP simulations on the following three stationary models and their Fourier transform respectively:
\begin{enumerate}
	\item Gaussian model: for $\rho \ge 0$ and $0<\alpha \le \sqrt{\frac{1}{\pi \rho}}$ %for $\alpha>0$ and $0 \le \rho \le \frac{1}{\pi \alpha^2}$
	\begin{eqnarray}
		\mathbb{K}(x,y) &=& \rho \exp\left(-\frac{1}{\alpha^2}\|x-y\|^2\right)\\
		\lambda_k &=& \pi\alpha^2\rho e^{-\pi^2\alpha^2\|k\|^2}
	\end{eqnarray}
	\item Mat{\'e}rn model: for $\rho \ge 0$, $\nu>0$ and $0<\alpha \le \sqrt{\frac{1}{4\pi \nu \rho}}$ %for $\alpha>0$, $\nu>0$ and $0 \le \rho \le \frac{1}{4\pi\alpha^2\nu}$
	\begin{eqnarray}
		\mathbb{K}(x,y) &=& \rho \frac{2^{1-\nu}}{\Gamma(\nu)\alpha^{\nu}}\|x-y\|^\nu \textbf{K}_\nu(\frac{1}{\alpha}\|x-y\|)\label{eq:W-HKernel}\\
		\lambda_k &=& 4\pi\alpha^2\rho \frac{\nu}{(1+4\pi^2\alpha^2\|k\|^2)^{1+\nu}}
	\end{eqnarray}
	\item Cauchy model (generalized): for $\rho \ge 0$, $\nu>0$ and $0<\alpha \le \sqrt{\frac{\nu}{\pi \rho}}$ %for $\alpha>0$, $\nu>0$ and $0 \le \rho \le \frac{\nu}{\pi\alpha^2}$
	\begin{eqnarray}
		\mathbb{K}(x,y) &=& \frac{\rho}{(1+\frac{1}{\alpha^2}\|x-y\|^2)^{1+\nu}}\\
		\lambda_k &=& \frac{2^{1-\nu}\pi\alpha^2\rho}{\Gamma(1+\nu)}\|2\pi\alpha k\|^{\nu}\textbf{K}_\nu(\|2\pi\alpha k\|)\label{eq:CauchyEigen}
	\end{eqnarray}
\end{enumerate}
where $\textbf{K}_\nu$ is the modified Bessel function of the second kind. The integral representation of the function is:
		\begin{equation*}
			\textbf{K}_\nu(z) = \frac{\left(\frac{z}{2}\right)^\nu\Gamma(\frac{1}{2})}{\Gamma(\nu+\frac{1}{2})}\int_1^\infty{e^{-zt}(t^2-1)^{\nu-\frac{1}{2}}dt}
		\end{equation*}
As $z$ tends to $0$, we have $\textbf{K}_\nu(z)$ tends to infinity and $\lim_{z\rightarrow 0} z^\nu\textbf{K}_\nu(z) = \frac{\Gamma(\nu)}{2^{1-\nu}}$.
In Mat{\'e}rn model, we have $\mathbb{K}(x,x)=\rho$, and in Cauchy model, we have $\lambda_0= \frac{\pi\alpha^2\rho}{\nu}$.
See other representations of $\textbf{K}_\nu(z)$ in \cite{Gradshteyn(2007)}. 

From Lemma~\ref{lemma:CPIupperbound} and (\ref{eq:Jmatrix}), we set the upper bound of the Papangelou intensity: 
\begin{equation}
	H = J_\Lambda(x,x) = \sum_{k\in \mathbb{Z}^2}{\frac{\lambda_k}{1-\lambda_k}}
\end{equation}

In practice we are unable to compute the infinite sum of the Fourier expansions. From (\ref{eq:DPPmeanVar}), a rule of thumb will be choosing a constant $N$ large enough such that that:
\begin{equation}
	\sum_{-N\le i,j\le N}{\lambda_{i,j}^{\Delta}}\approx \rho 
\end{equation}

For the given models, we compute their pair correlation function $g(r)$ to investigate the reliability of the simulations, where $r:=\|x-y\|^2$. Given any $2$ points $x$ and $y$ in a configuration, their
pair correlation function is given as:
\begin{equation*}
	g(r) := \frac{\rho_2(x,y)}{\rho_1(x)\rho_1(y)}= 1-\frac{\mathbb{K}(x,y)\mathbb{K}(y,x)}{\mathbb{K}(x,x)\mathbb{K}(y,y)}
\end{equation*}
and respectively for:
\begin{enumerate}
	\item Gaussian model:
	\begin{eqnarray}
		g(r)&=&1-e^{-2(r/\alpha)^2},
	\end{eqnarray}
	\item Mat{\'e}rn model:
	\begin{eqnarray}
		g(r)&=&1-[2^{1-\nu}(\frac{r}{\alpha})^\nu\textbf{K}_\nu\left(\frac{r}{\alpha}\right)/\Gamma(\nu)]^2,
	\end{eqnarray}
	\item Cauchy model:
	\begin{eqnarray}
		g(r)&=&1-(1+(r/\alpha)^2)^{-2\nu-2}.
	\end{eqnarray}
\end{enumerate}

In this paper, we fixed $N$ such that the sum is at least $99.9\%$ of $\rho$. 
%For Gaussian model, $N=20$; Mat{\'e}rn Model, $N=30$ and Cauchu model, $N=25$.
In Figure~\ref{fig:PMF}, we compare the distribution of the size of DPP simulated with its \textit{actual} distribution given by a Poisson-binomial distribution, see \cite{Hough(2006)}.
Here, we use the term \textit{actual} for the truncated Fourier expansion to approximate the kernel. From \cite{Fernandez(2010)}, the probability mass function (PMF) of the
Poisson-binomial distribution can be written in the form of discrete Fourier transform as:
\begin{equation*}
	\Pr (|X|=n)=\frac{1}{N+1}\sum_{k=0}^{N}{{{e^{\frac{-2\pi i kn}{N+1}}}}\prod_{m=1}^{N}{\left( p_me^{\frac{2\pi i k}{N+1}}+(1-p_m) \right)}}
\end{equation*}

\begin{figure}[htbp] %htbp
	\centering
        \includegraphics[width=0.5\textwidth]{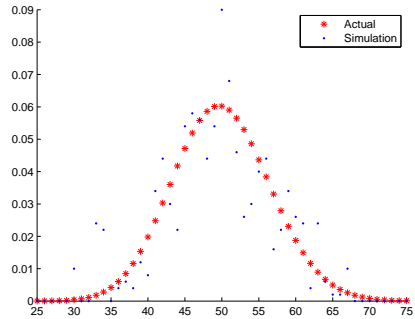}
				\caption{PMF of $500$ Gaussian models with $\rho=50$ and $\alpha =0.04$.}
   \label{fig:PMF}
\end{figure}

Figure \ref{fig:PCF} is a comparison of the theoretical pair correlation function $g(r)$ with the simulated DPP results for $\rho=50$ and $\alpha$ is fixed to be $\alpha_{\max}/2$.
For both Mat{\'e}rn Model and Cauchy Model, $\nu=5$. With regards to the the error in $g(r)$ when $r$ tends to $0$, we lost some `harmonic' in the Fourier approximations when computing the Papangelou intensity.

\begin{figure}[htbp] %htbp
	\centering
        \subfigure[Gaussian Model]{%
            \label{fig:GaussianPCF}
            \includegraphics[width=0.48\textwidth]{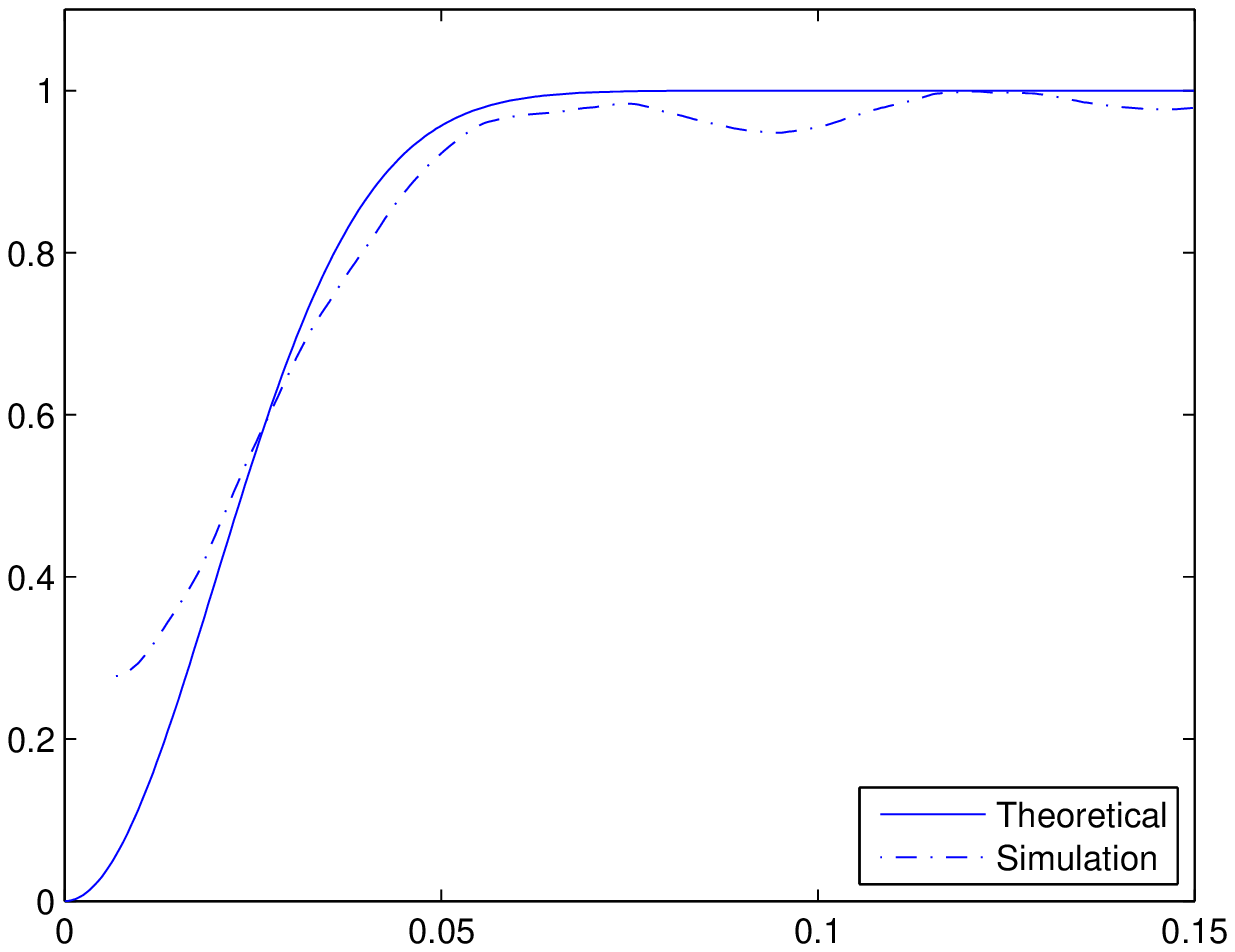}
        }%
        \subfigure[Mat{\'e}rn Model]{%
           \label{fig:MaternPCF}
           \includegraphics[width=0.48\textwidth]{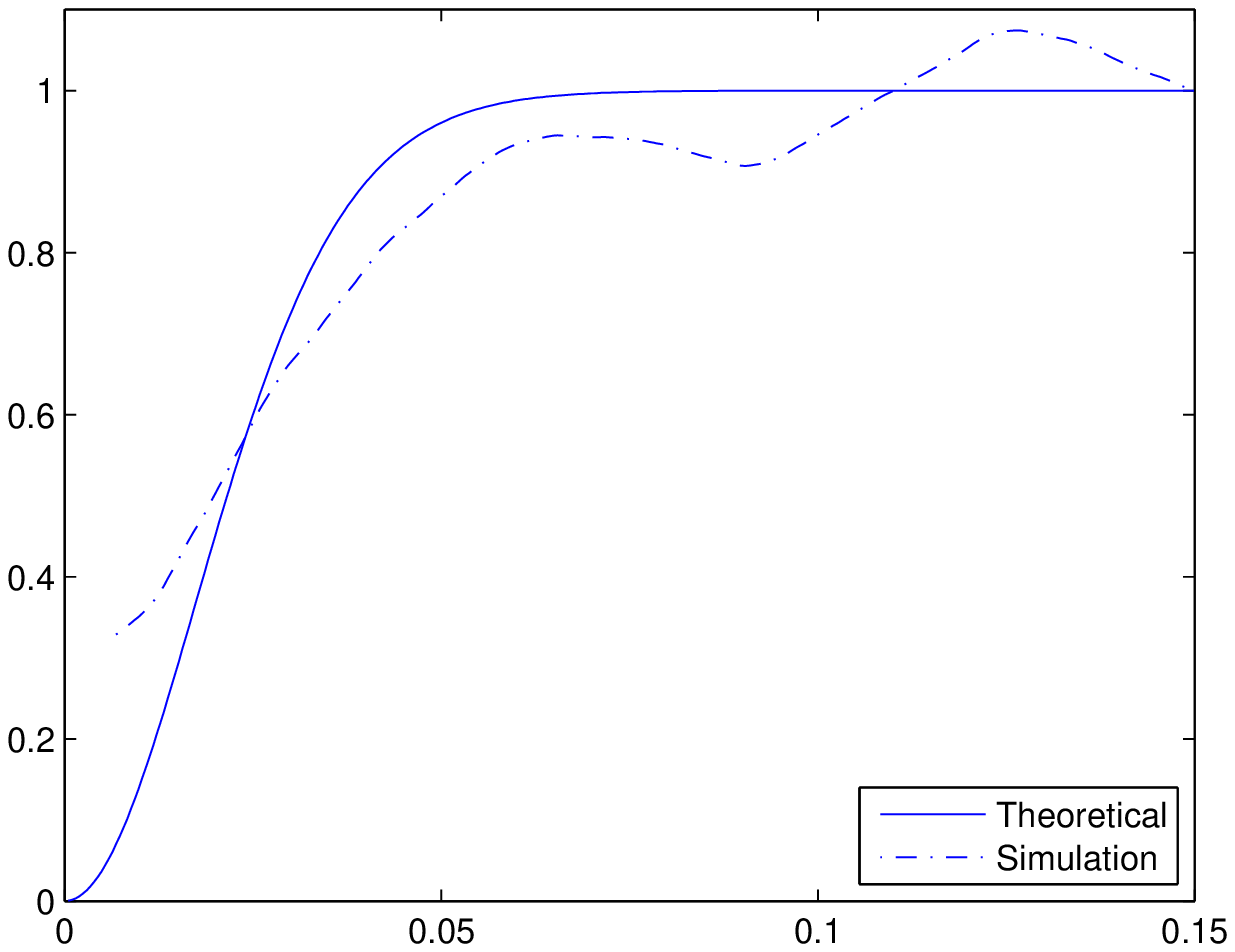}
        }\\ %  ------- End of the first row ----------------------%
        \subfigure[Cauchy model]{%
            \label{fig:CauchyPCF}
           \includegraphics[width=0.48\textwidth]{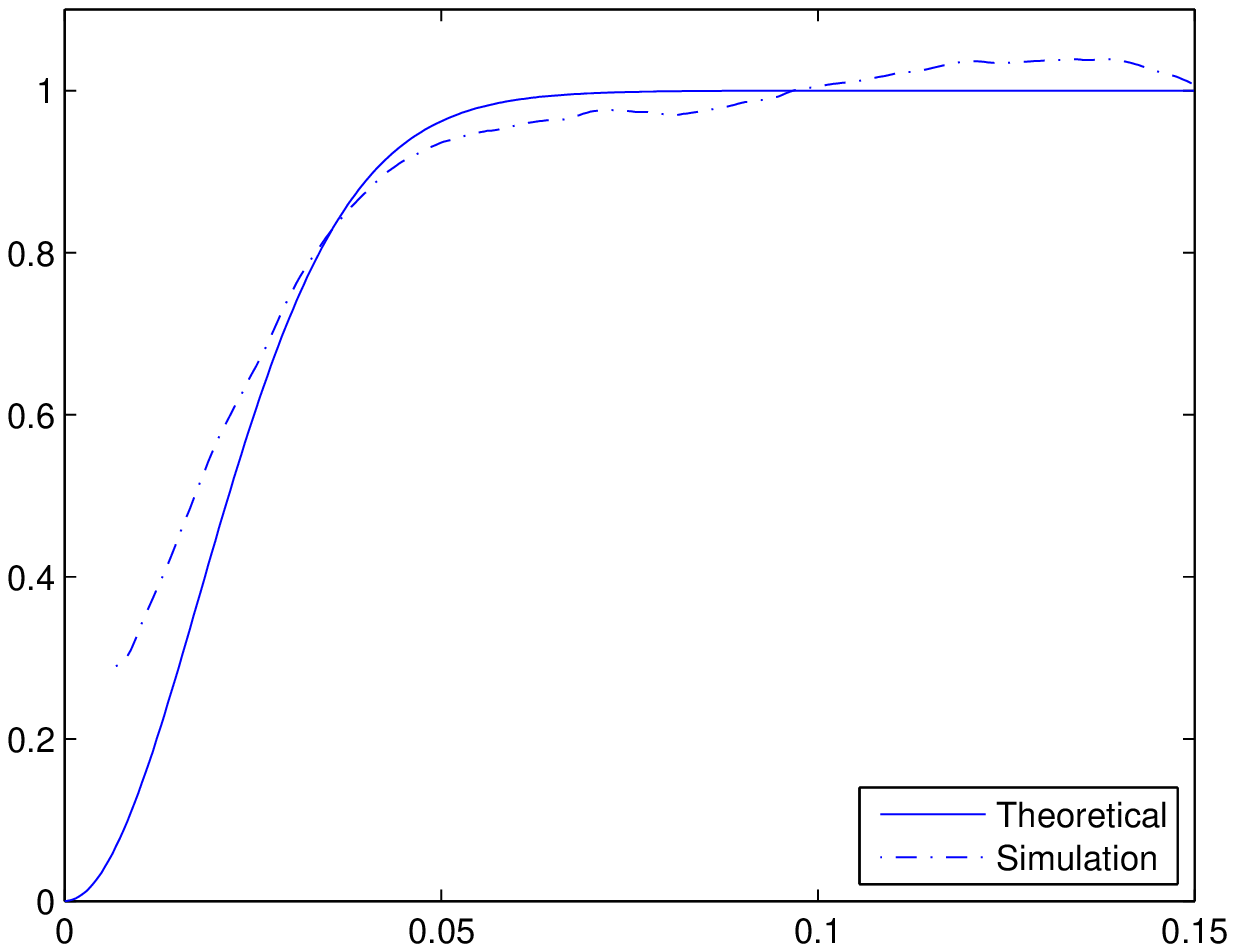}
        }%
				\caption{Pair Correlation Function.}
   \label{fig:PCF}
\end{figure}

Figure \ref{fig:CoalescenceTime} and Figure \ref{fig:StoppingTime} show the distributions of the coalescence time of $L_t$ and $U_t$ and the stopping time of the algorithm for $500$ Gaussian model DPP simulations with $\rho=50$. The stopping time refer to the time $-n$ required to simulate backwards such that coalescence of $L_t$ and $U_t$ occur at time $t=0$. Heuristically, we have shown that the stopping time is much lower than the upper bound $H|\Lambda|\log H|\Lambda|$, ($H\approx 57.5$).

\begin{figure}[htbp] %htbp
	\centering
        \subfigure[Coalescence time of $L_t$ and $U_t$]{%
            \label{fig:CoalescenceTime}
            \includegraphics[width=0.48\textwidth]{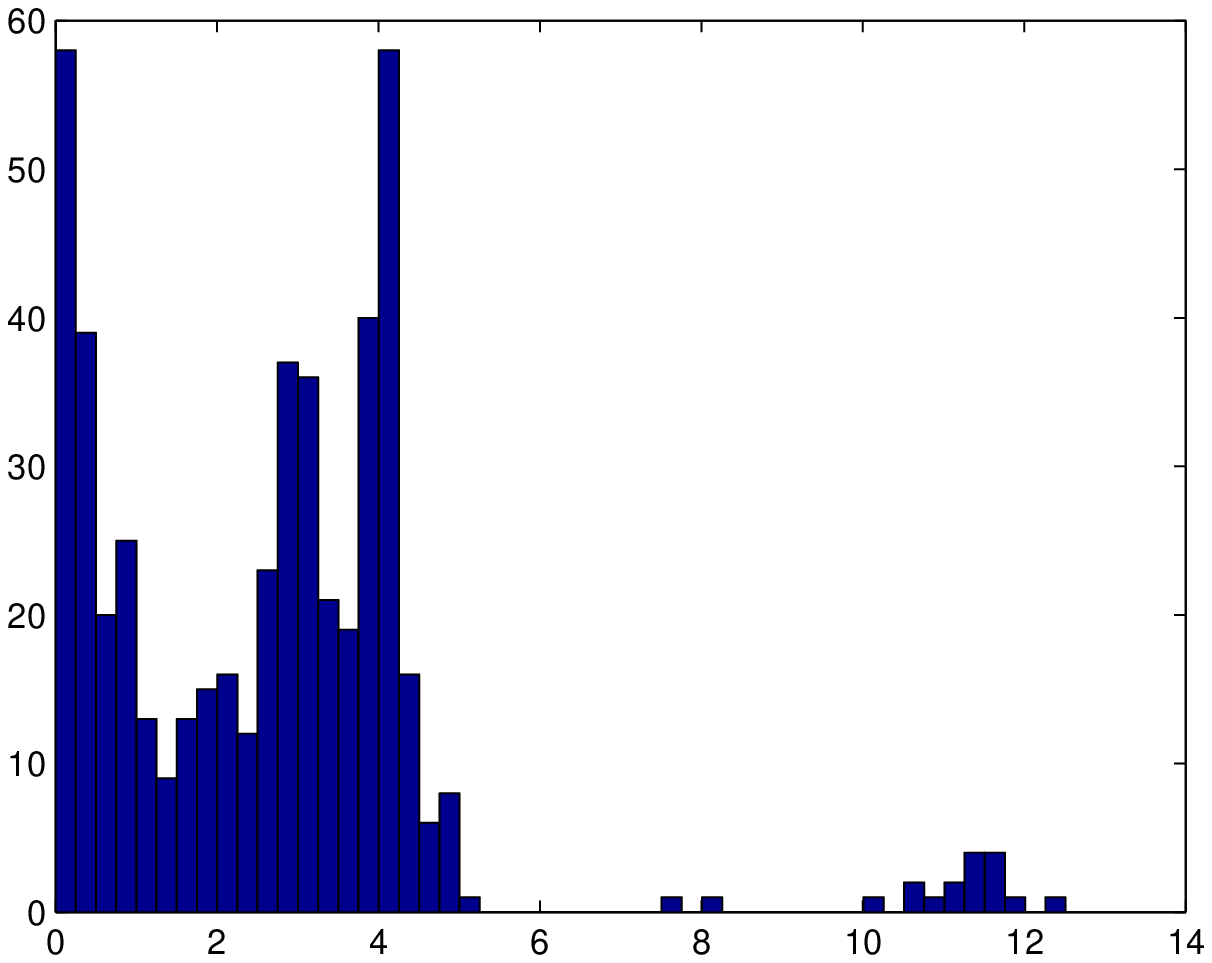}
        }%
        \subfigure[Stopping time]{%
           \label{fig:StoppingTime}
           \includegraphics[width=0.48\textwidth]{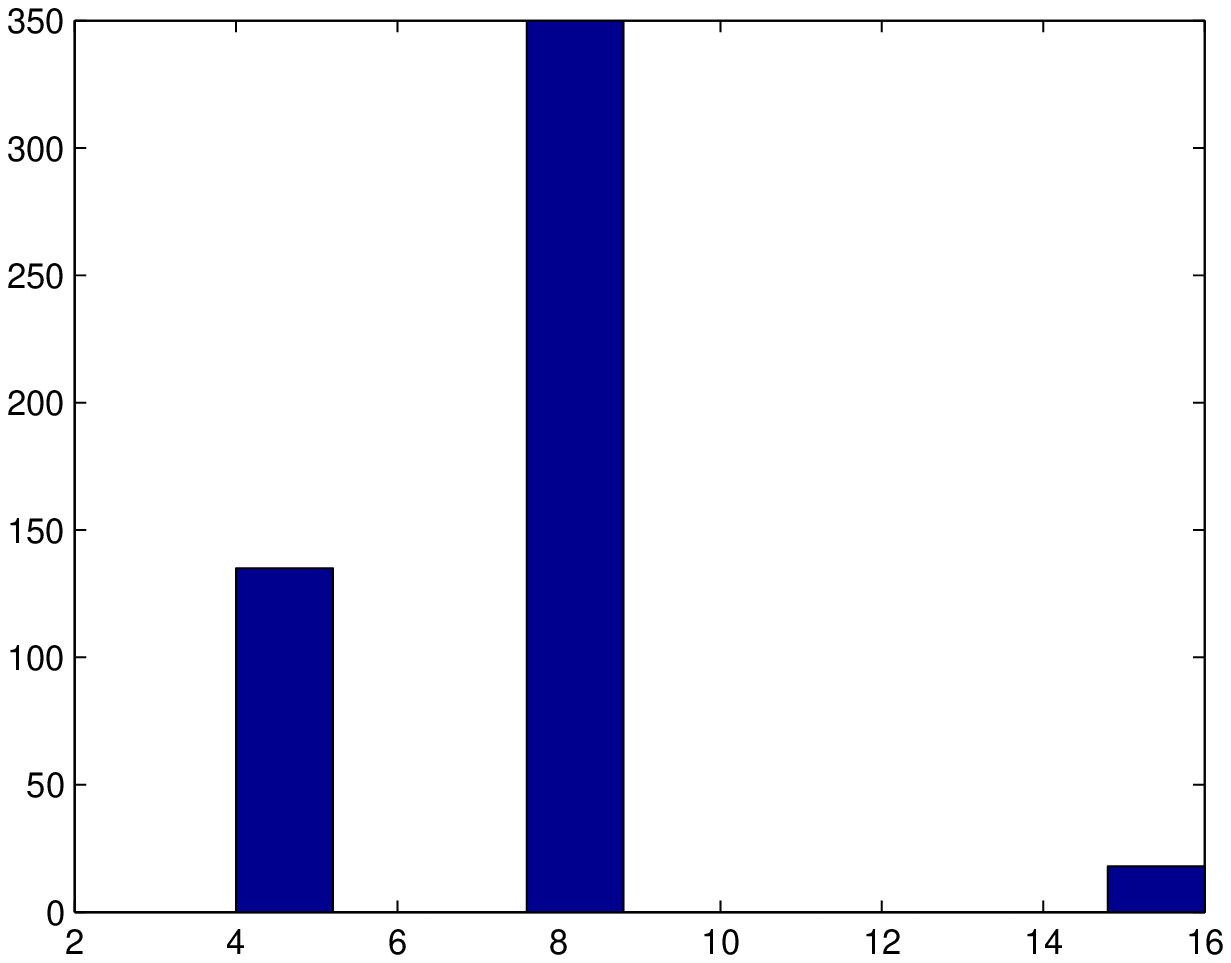}
        }\\%
				\caption{Histograms of coalescence time of $L_t$ and $U_t$ and the stopping time on $500$ Gaussian models with $\rho=50$ and $\alpha =0.04$.}
   \label{fig:Time}
\end{figure}

% \appendix
% \section{Appendix section}\label{app}

%\begin{thebibliography}{9}
\providecommand{\bysame}{\leavevmode\hbox to3em{\hrulefill}\thinspace}
\providecommand{\MR}{\relax\ifhmode\unskip\space\fi MR }
% \MRhref is called by the amsart/book/proc definition of \MR.
\providecommand{\MRhref}[2]{%
  \href{http://www.ams.org/mathscinet-getitem?mr=#1}{#2}
}
\providecommand{\href}[2]{#2}

%\bibliographystyle{amsplain}
%\bibliography{references}

\begin{thebibliography}{10}

\bibitem{Daley}
D.J. Daley and D.~Vere-Jones, \emph{An introduction to the theory of point
  processes}, 2 ed., vol. I: Elementary Theory and Methods, Springer New York,
  2003.

\bibitem{Daley2}
\bysame, \emph{An introduction to the theory of point processes}, 2 ed., vol.
  II: General Theory and Structure, Springer New York, 2008.

\bibitem{Fernandez(2010)}
M.~Fernandez and S.~Williams, \emph{Closed-form expression for the
  poisson-binomial probability density function}, Aerospace and Electronic
  Systems, IEEE Transactions on \textbf{46} (2010), no.~2, 803--817.

\bibitem{Georgii}
Hans-Otto Georgii and HyunJae Yoo, \emph{Conditional intensity and gibbsianness
  of determinantal point processes}, Journal of Statistical Physics
  \textbf{118} (2005), no.~1-2, 55--84.

\bibitem{Hough(2006)}
John~Ben Hough, Manjunath Krishnapur, Yuval Peres, and B{\'a}lint Vir{\'a}g,
  \emph{Determinantal processes and independence.}, Probability Surveys
  \textbf{3} (2006), 206--229.

\bibitem{Peres}
\bysame, \emph{Zeros of gaussian analytic functions and determinantal point
  processes}, vol.~51, University Lecture Series, American Mathematical
  Society, Providence, RI, 2009.

\bibitem{Gradshteyn(2007)}
Alan Jeffrey and Daniel Zwillinger, \emph{Gradshteyn and ryzhik's table of
  integrals, series, and products}, 7 ed., Academic Press, 2007.

\bibitem{Kendall(1998)}
Wilfrid~S. Kendall, \emph{Perfect simulation for the area-interaction point
  process}, Probability Towards 2000 (L.~Accardi and C.C. Heyde, eds.), Lecture
  Notes in Statistics, vol. 128, Springer New York, 1998, pp.~218--234.

\bibitem{KendallMoller(2000)}
Wilfrid~S. Kendall and Jesper M{\o}ller, \emph{Perfect simulation using
  dominating processes on ordered spaces, with application to locally stable
  point processes}, Advances in Applied Probability \textbf{32} (2000), no.~3,
  844--865.

\bibitem{LMR12}
F.~Lavancier, J.~M{\o}ller, and E.~Rubak, \emph{Determinantal point process
  models and statistical inference.}, submitted (arxiv:1205.4818) (2012),
  1--46.

\bibitem{Papangelou}
F.~Papangelou, \emph{The conditional intensity of general point processes and
  an application to line processes}, Zeitschrift für
  Wahrscheinlichkeitstheorie und Verwandte Gebiete \textbf{28} (1974), no.~3,
  207--226.

\bibitem{ProppWilson(1996)}
James~Gary Propp and David~Bruce Wilson, \emph{Exact sampling with coupled
  markov chains and applications to statistical mechanics}, Random Structures
  \& Algorithms \textbf{9} (1996), no.~1-2, 223--252.

\bibitem{Robert(2003)}
Philippe Robert, \emph{Stochastic networks and queues}, Stochastic Modelling
  and Applied Probability Series, Springer-Verlag, New York, 2003.

\bibitem{ShiraiTaka(2003)}
Tomoyuki Shirai and Yoichiro Takahashi, \emph{Random point fields associated
  with certain fredholm determinants ii: Fermion shifts and their ergodic and
  gibbs properties.}, Ann. Probab. \textbf{31} (2003), no.~3, 1533--1564.

\bibitem{Soshnikov(2000)}
A.~B. Soshnikov, \emph{Determinantal random point fields.}, Russian
  Mathematical Surveys \textbf{55} (2000), no.~5, 923--975.

\end{thebibliography}
%\end{thebibliography}
\end{document}